\documentclass{article}
\usepackage[english]{babel}
\usepackage{url}
\usepackage{parskip}
\usepackage{caption}
\usepackage{subcaption}
\usepackage[usenames,dvipsnames]{color}
\usepackage[utf8]{inputenc}
\usepackage{float}
\usepackage{amssymb}
\usepackage{amsthm}
\usepackage[makeroom]{cancel}
\usepackage{amsmath}
\usepackage{mathtools}
\usepackage{mathrsfs}
\usepackage{amsfonts}
\usepackage{placeins}
\usepackage{graphicx}
\usepackage{geometry}
\usepackage{comment}
\usepackage{wrapfig}
\usepackage{hyperref}
\usepackage[useregional]{datetime2}
\usepackage{tikz}
\usepackage{enumitem}
\usepackage{algorithm}
\usepackage{algorithmic}
\usepackage{bm}
\usepackage{cite}

\newcommand{\R}{\mathbb{R}}
\newcommand{\N}{\mathcal{N}}

\newtheorem{theorem}{Theorem}[section]

\theoremstyle{definition}
\newtheorem{definition}{Definition}[section]
\theoremstyle{remark}
\newtheorem*{remark}{Remark}

\allowdisplaybreaks

\title{\LARGE \bf
Accelerated Performance and Accelerated Learning with Discrete-Time High-Order Tuners\thanks{This work is supported by the Boeing Strategic University Initiative.}
}

\author{Yingnan Cui\thanks{Y. Cui and A.M. Annaswamy are with the Department of Mechanical Engineering, Massachusetts Institute of Technology, Cambridge, MA, 02139.} and Anuradha M. Annaswamy
}

\date{}
\begin{document}
\allowdisplaybreaks

\maketitle
\thispagestyle{empty}
\pagestyle{empty}

\begin{abstract}

  We consider two high-order tuners that have been shown to have accelerated performance, one based on Polyak's heavy ball method and another based on Nesterov's acceleration method. We show that parameter estimates are bounded and converge to the true values exponentially fast when the regressors are persistently exciting. Simulation results corroborate the accelerated performance and accelerated learning properties of these high-order tuners in comparison to algorithms based on normalized gradient descent.

\end{abstract}

\section{Introduction}
\label{sec:intro}

Adaptive control is dedicated to online decision making and parameter learning in dynamic systems in real-time \cite{Narendra2005,Ioannou1996,Annaswamy2021hist,Goodwin_1984}. The problem often consists of addressing two types of errors, one related to performance, and another related to parameter learning. The goal is to have both the performance error and parameter error converge to zero in real time. As the complexity of the dynamic systems increases and the performance specifications become more stringent, it is of importance that both of these convergences are fast. This paper addresses both of these properties in the context of a high-order tuner.


High-order tuners have their start in \cite{Morse_1992}, where a class of continuous time dynamic systems with parametric uncertainties was considered and an adaptive law that allowed the generation of parameter estimates using a high-order tuner was proposed rather than a standard gradient algorithm that is of first-order.  Parameter learning with these high-order tuners was addressed in \cite{Ortega93,Annaswamy21}. Robustness properties using these tuners was addressed in \cite{Annaswamy2003}. All of these discussions have focused entirely on continuous-time dynamic systems.

In parallel with these developments, a body of work has been ongoing in the optimization community to address accelerated convergence of the performance error \cite{Polyak64,Nesterov_1983}, and studied at length in machine learning and optimization \cite{Su14,Wibisono_2016}, and identification \cite{Gaudio20AC, Gaudio2021}. The idea here is to identify methods by which the performance error can converge to zero faster. As this has broad implications on a large number of problems in control, machine learning, and optimization, the impact of success in these investigations can be a significant one. The focus in all of these problems however is only on the performance error, often embodied by an overall loss function. They do not focus on parameter learning or the speed of parameter convergence, which is the focus of this paper.

In this paper, we consider a class of discrete-time nonlinear systems whose parameters are constant and unknown. The goal is to design an estimator that will learn the parameters using real-time data. It is shown that under conditions of persistent excitation, high-order tuners can be utilized to ensure accelerated learning, i.e., exponential convergence of the parameter estimates to their true value. This problem has been addressed at length in continuous-time both using standard gradient based adaptive laws \cite{Narendra2005,Ioannou1996} and high-order tuners \cite{Ortega93, Annaswamy21}. Parameter convergence has been addressed in discrete-time only using gradient laws in \cite{Goodwin_1984, Anderson_1982} but not using high-order tuners. Together with the results in \cite{Polyak64, Nesterov_1983, Gaudio20AC, Gaudio2021}, this paper lays the foundation for new algorithms that can provide both accelerated performance and accelerated learning, and represents its main contribution. Our focus is on high-order tuners that are based on two popular methods, one based on the Heavy Ball (HB) method \cite{Polyak64} that includes a momentum-like terms, and the other based on Nesterov's algorithm (NA) \cite{Nesterov_1983} that includes terms based on both momentum and acceleration. While these approaches can be used readily for static problems of decision-making, when dynamic features are present, there needs to be significant variations in the underlying algorithm \cite{Gaudio20AC,Gaudio2021}. As shown in these papers, appropriate variations need to be made when the underlying regressors vary with time, using which a bounded parameter estimation approach can be derived. No discussions were carried out, however, in \cite{Gaudio20AC,Gaudio2021} regarding accelerated learning.



The main challenge that a high-order tuner introduces for establishing accelerated learning in parameters is the presence of additional state variables that introduces a filtering action between the exogeneous signal that is persistently exciting and the parameter that is to be estimated. These state variables have to be shown to behave in a way such that the excitation is transmitted through them without any attenuation, thereby allowing the estimate to continue converging to the true value. This property is even more difficult to establish in a discrete-time system than a continuous-time one due to the underlying support set properties and challenges in ensuring that the step sizes in the updates remain bounded. We successfully addressed these challenges through novel tools that leverage both properties of the high-order tuner and those of persistent excitation. This is in contrast to its continuous-time counterparts where exponential convergence of the parameter estimates to their true values is obtained by appealing primarily to specific properties of persistent excitation.

The paper is organized as follows. Section \ref{sec:problem-statement} presents problem statement. Section \ref{sec:heavy-ball} and section \ref{sec:nesterov} contain the main results of the paper. Section \ref{sec:heavy-ball} shows the exponential convergence of the HB method. Section \ref{sec:nesterov} shows the exponential convergence of the NA algorithm. We show simulation results in Section \ref{sec:sim} and provide concluding remarks in Section \ref{sec:conclusion}.
\section{Problem Statement}
\label{sec:problem-statement}
We consider a class of discrete-time nonlinear plant models of the form
\begin{equation}
  \label{eq:10}
    y_k = -\sum_{i=1}^{n}a_{i}^*y_{k-i} + \sum_{j=1}^{m}b_{j}^*u_{k-j-d} + \sum_{\ell=1}^pc^*_{\ell}f_\ell(y_{k-1}, \ldots, y_{k-n}, u_{k-1-d}, \ldots, u_{k-m-d}),
\end{equation}
where $a_{i}^*$, $b_{j}^*$ and $c_{\ell}^*$ are unknown parameters that are constant and need to be identified, and $d$ is a known time-delay. The function $f_\ell$ is an analytic function and is assumed to be such that the system in \eqref{eq:10} is bounded-input-bounded-output (BIBO) stable. Denote $z_{k-1} = [y_{k-1}, \ldots, y_{k-n}]^\top$ and $v_{k-d-1} = [u_{k-1-d}, \ldots, u_{k-m-d}]^\top$. We rewrite \eqref{eq:10} in the form of a linear regression
\begin{equation}
  \label{eq:11}
  y_k = \phi_k^\top\theta^*,
\end{equation}
where $\phi_k = [z_{k-1}^\top, v_{k-d-1}^\top, f_1(z_{k-1}^\top, v_{k-d-1}^\top), \ldots, \allowbreak f_p(z_{k-1}^\top, v_{k-d-1}^\top)]^\top$ is a regressor determined by exogenous signals and $\theta^* = [a_{1}^*, \ldots, a_{n}^*, b_{1}^*, \ldots, b_{m}^*,c_{1}^*, \ldots, c_{\ell}^*]^\top$ is the underlying unknown parameter vector. We propose to identify the parameter $\theta^*$ as $\theta_k$ using an estimator
\begin{equation}
  \hat{y}_k = \phi_k^\top\theta_k,
  \label{eq:estimate}
\end{equation}
which leads to a prediction error
\begin{equation}
  \label{eq:12}
  e_{y,k} = \phi_k^{\top}\tilde{\theta}_k,
\end{equation}
where $e_{y,k} = \hat{y}_k - y_k$ is the output prediction error and $\tilde{\theta}_k = \theta_k - \theta^*$ is the parameter error. The goal of parameter identification is to design an iterative procedure such that the parameter error $\|\tilde{\theta}_k\|$ converges to zero exponentially fast.

The iterative procedure for estimating the parameters is based on a squared loss function,
\begin{equation}
  \label{eq:7}
  L_k(\theta_k) = \frac{1}{2}e_{y,k}^2 = \frac{1}{2}\tilde{\theta}_k^\top\phi_k\phi_k^\top\tilde{\theta}_k,
\end{equation}
where the subscript $k$ in $L_k$ denotes $k$th iteration. In the literature, a normalized gradient descent algorithm has been shown to be stable although having a slow convergence rate \cite{Goodwin_1984}
\begin{equation}
  \label{eq:gd}
  \theta_{k+1} = \theta_k - \alpha\frac{\nabla L_k(\theta_k)}{\N_k}, \quad 0 < \alpha < 2,
\end{equation}
where $\N_k$ is a normalizing signal and is defined as $\N_k = 1 + \|\phi_k\|^2$. The following definitions will be utilized for proving the main results.

\begin{definition}
The regressor $\phi_k$ is said to satisfy the persistent excitation (PE) condition over an interval $\Delta T$, if for all unit vectors $w\in\R^n$,
\begin{equation}
  \label{eq:pe}
  \frac{1}{\Delta T}\sum_{i=k-\Delta T}^{k-1}\left\|\phi_k^\top w\right\| \geq \epsilon.
\end{equation}
\label{def:pe}
\end{definition}
\begin{definition}[From \cite{Luenberger1997}]
  For any fixed $p\in[1,\infty)$, a sequence of scalars $\xi=\{\xi_0, \xi_1, \ldots\}$ is defined to belong to $\ell_p$ if
  \begin{equation}
    \label{eq:2}
    \|\xi\|_\infty \equiv \left(\lim_{k\rightarrow\infty}\sum_{i=0}^k\|\xi_i\|^p\right)^{1/p} < \infty.
  \end{equation}
  When $p=\infty$, $\xi\in\ell_\infty$ if
  \begin{equation}
    \label{eq:3}
    \|\xi\|_{\ell_\infty} \equiv \sup_{i\geq 0}\|\xi_i\| < \infty
  \end{equation}
\end{definition}

\section{Main Result 1: Accelerated Learning with Heavy Ball Method}
\label{sec:heavy-ball}
The idea behind the Heavy Ball method can be explained as follows. Rather than using only the past iterate $\theta_k$ to determine $\theta_{k+1}$, the Heavy Ball method uses the past two iterates $\theta_k$ and $\theta_{k-1}$ so that an additional momentum term may contribute to an accelerated convergence of the loss function. This takes the form of a \textit{higher}-order tuner of the form
\begin{equation}
  \theta_{k+1} = \theta_k-\bar\gamma\frac{\nabla L_k{(\theta_{k})}}{\N_k} + \bar\beta (\theta_k-\theta_{k-1}),
  \label{eq:HT-HB}
\end{equation}
where $\bar\gamma$ and $\bar\beta$ are hyperparameters and the last term corresponds to the momentum addition. This high-order tuner can be rewritten in the form of two first-order iterates
\begin{equation}
\begin{split}
  \label{eq:hb}
  \vartheta_{k+1} &= \vartheta_k - \gamma\frac{\nabla L_k(\theta_{k+1})}{\N_k},\\
  \theta_{k+1} &= \theta_k - \beta(\theta_k - \vartheta_k),
\end{split}
\end{equation}
where $\beta$ and $\gamma$ are the positive constants that will be suitably chosen. We denote this as an HB algorithm. It is easy to show that the estimates $\vartheta_k$ and $\theta_k$ are bounded using the following Lyapunov function:
\begin{equation}
  \label{eq:lyap}
  V_k = \frac{1}{\gamma}\|\vartheta_k - \theta^*\|^2 + \frac{1}{\gamma}\|\theta_k - \vartheta_k\|^2
\end{equation}
for all $0 < \beta < 2$ and $0 < \gamma \leq \frac{\beta(2 - \beta)}{16}$ \cite{Gaudio20AC}. In what follows, we show that the HT in \eqref{eq:hb} guarantees accelerated learning. A few parameters are defined first.

Let
\begin{equation*}
  c_1 = \frac{11}{8}, \quad c_2= \frac{21}{32},
\end{equation*}
\begin{equation*}
  \epsilon_1 = \frac{\epsilon}{\max_k\{\sqrt{\N_k}\}},
\end{equation*}
\begin{equation}
0 < \lambda < 1,
\label{eq:lambda}
\end{equation}
\begin{equation}
  0 < \eta < \frac{\epsilon_1}{\gamma|1 - \beta|}.
  \label{eq:eta}
\end{equation}
We define
\begin{equation}
  \mu = \min\{\mu_1, \mu_2, \mu_3\},
  \label{eq:mu}
\end{equation}
where
\begin{align*}
  \mu_1 &= \frac{c_1\lambda\gamma\eta^2}{\Delta T},\\
  \mu_2 &= \frac{c_2\Delta T(\epsilon_1 - \gamma\eta|1 - \beta|)^2\lambda\gamma}{(1+\gamma\Delta T)^2},\\
  \mu_3 &= c_1\frac{1}{\Delta T}(1 - \lambda)\gamma.
\end{align*}
\begin{theorem}
\label{theo:1}
  If the regressor $\phi_k$ satisfies the definition in \eqref{eq:pe}, with $0 < \beta < 2$ and $0 < \gamma \leq \frac{\beta(2 - \beta)}{8}$, the update law in \eqref{eq:hb} will result in (i) $\vartheta_k - \theta^* \in
 \ell_\infty$, $\theta_k - \vartheta_k\in \ell_\infty$, and (ii) $V_k \leq \exp\left(-\mu\left\lfloor\frac{k}{\Delta T}\right\rfloor\right)V_0$, where $\mu$ is defined in \eqref{eq:mu}.
\end{theorem}
Theorem \ref{theo:1} (i) establishes boundedness of the parameter estimates and Theorem \ref{theo:1} (ii) shows that exponential convergence of the parameter error towards zero occurs if $\phi_k$ is persistently exciting. In order to prove these results, we will examine the behavior of the parameter estimates over an interval $\Delta T$ over which $\phi_k$ is persistently exciting.

\begin{proof}
  Expanding $\Delta V_k := V_{k +1} - V_k$, we have
\begin{align*}
  &\Delta V_k \\
  &= \frac{1}{\gamma}\|\vartheta_{k+1} - \theta^*\|^2 + \frac{1}{\gamma}\|\theta_{k+1} - \vartheta_{k+1}\|^2 - \frac{1}{\gamma}\|\vartheta_k - \theta^*\|^2 - \frac{1}{\gamma}\|\theta_k - \vartheta_k\|^2\\
  &= \frac{1}{\gamma}\left\|\vartheta_k - \gamma\frac{\nabla L_k(\theta_{k+1})}{\N_k} - \theta^*\right\|^2 - \frac{1}{\gamma}\|\vartheta_k - \theta^*\|^2 + \frac{1}{\gamma}\left\|\theta_k - \beta(\theta_k - \vartheta_k) - \vartheta_k + \frac{\gamma}{\N_k}\nabla L_k(\theta_{k+1})\right\|^2\\
  &\quad - \frac{1}{\gamma}\|\theta_k - \vartheta_k\|^2\\
  &= \frac{\gamma}{\N_k^2}\|\nabla L_k(\theta_{k+1})\|^2 - \frac{2}{\N_k}(\vartheta_k - \theta^*)^\top\nabla L_k(\theta_{k+1}) + \frac{1}{\gamma}\|\theta_k - \vartheta_k\|^2 - \frac{1}{\gamma}\|\theta_k - \vartheta_k\|^2\\
  &\quad -\frac{\beta(2 - \beta)}{\gamma}\|\theta_k - \vartheta_k\|^2 + \frac{2}{\N_k}(1 - \beta)(\theta_k - \vartheta_k)^\top \nabla L_k(\theta_{k+1}) + \frac{\gamma}{\N_k^2}\|\nabla L_k(\theta_{k+1})\|^2\\
  &= \frac{2\gamma}{\N_k^2}\|\nabla L_k(\theta_{k+1})\|^2 - \frac{2}{\N_k}(\theta_{k+1} - \theta^*)^\top \nabla L_k(\theta_{k+1}) - \frac{\beta(2 - \beta)}{\gamma}\|\theta_k - \vartheta_k\|^2\\
  &\quad +\frac{2}{\N_k}(1 - \beta)(\theta_k - \vartheta_k)^\top\nabla L_k(\theta_{k+1}) - \frac{2}{\N_k}(\vartheta_k - \theta_{k+1})^\top \nabla L_k(\theta_{k+1})\\
  &= -2\left(1 - \frac{\gamma\phi_k^\top\phi_k^\top}{\N_k}\right)\frac{\tilde{\theta}_{k+1}^\top\nabla L_k(\theta_{k+1})}{\N_k} - \frac{\beta(2 - \beta)}{\gamma}\|\theta_k - \vartheta_k\|^2 + \frac{4}{\N_k}(1 - \beta)(\theta_k - \vartheta_k)^\top\nabla L_k(\theta_{k+1})\\
  &\leq -\frac{11}{8}\|\theta_k - \vartheta_k\|^2 + \frac{1}{\N_k}\Bigg\{-\|\tilde{\theta}_{k+1}^\top\phi_k\|^2 - 4\|\phi_k\|^2\|\theta_k - \vartheta_k\|^2 + 4\|\theta_k - \vartheta_k\|\|\phi_k\|\|\tilde{\theta}_{k+1}^\top\phi_k\|\\
  &\quad -4\|\theta_k - \vartheta_k\|^2 - \frac{21}{8}\|\phi_k\|^2\|\theta_k - \vartheta_k\|^2 - \frac{7}{8}\|\tilde{\theta}_{k+1}^\top\phi_k\|^2\Bigg\}\\
  &\leq \frac{1}{\N_k}\left\{-\left[\|\tilde{\theta}_{k+1}^\top\phi_k\| - 2\|\phi_k\|\|\theta_k - \vartheta_k\|\right]^2 - 4\|\theta_k - \vartheta_k\|^2 - \frac{21}{8}\|\phi_k\|^2\|\theta_k - \vartheta_k\|^2 - \frac{7}{8}\|\tilde{\theta}_{k+1}^\top\phi_k\|^2\right\}\\
  &= -\frac{11}{8}\|\theta_k - \vartheta_k\|^2 + \frac{1}{\N_k}\Bigg\{-\left[\|\tilde{\theta}_{k+1}^\top\phi_k\| - 2\|\phi_k\|\|\theta_k - \vartheta_k\|\right]^2 - 4\|\theta_k - \vartheta_k\|^2 - \frac{21}{8}\|\phi_k\|^2\|\theta_k - \vartheta_k\|^2 \\
  &\quad- \frac{7}{8}\|(\theta_{k+1} - \vartheta_k + \vartheta_k - \theta^*)^\top\phi_k\|^2\Bigg\}\\
  &= -\frac{11}{8}\|\theta_k - \vartheta_k\|^2 + \frac{1}{\N_k}\Bigg\{-\left[\|\tilde{\theta}_{k+1}^\top\phi_k\| - 2\|\phi_k\|\|\theta_k - \vartheta_k\|\right]^2 - 4\|\theta_k - \vartheta_k\|^2 - \frac{21}{8}\|\phi_k\|^2\|\theta_k - \vartheta_k\|^2\\
  &\quad- \frac{7}{8}\left\|(1 - \beta)(\theta_k - \vartheta_k)^\top\phi_k + (\vartheta_k - \theta^*)^\top\phi_k\right\|^2\Bigg\}\\
  &= -\frac{11}{8}\|\theta_k - \vartheta_k\|^2 + \frac{1}{\N_k}\Bigg\{-\left[\|\tilde{\theta}_{k+1}^\top\phi_k\| - 2\|\phi_k\|\|\theta_k - \vartheta_k\|\right]^2 - 4\|\theta_k - \vartheta_k\|^2 - \frac{21}{8}\|(\theta_k - \vartheta_k)^\top\phi_k\|^2\\
  &\quad- \frac{7}{8}\left[\left\|(1 - \beta)(\theta_k - \vartheta_k)^\top\phi_k\right\|^2 + \|(\vartheta_k - \theta^*)^\top\phi_k\|^2 + 2(1 - \beta)(\theta_k - \vartheta_k)^\top\phi_k\phi_k^\top(\vartheta_k - \theta^*)\right]\Bigg\}\\
  &= -\frac{11}{8}\|\theta_k - \vartheta_k\|^2 + \frac{1}{\N_k}\Bigg\{-\left[\|\tilde{\theta}_{k+1}^\top\phi_k\| - 2\|\phi_k\|\|\theta_k - \vartheta_k\|\right]^2 - 4\|\theta_k - \vartheta_k\|^2 - \frac{21}{8}\|(\theta_k - \vartheta_k)^\top\phi_k\|^2\\
  &\quad- \frac{7}{8}\left[4\left\|(1 - \beta)(\theta_k - \vartheta_k)^\top\phi_k\right\|^2 + \frac{1}{4}\|(\vartheta_k - \theta^*)^\top\phi_k\|^2 + 2(1 - \beta)(\theta_k - \vartheta_k)^\top\phi_k\phi_k^\top(\vartheta_k - \theta^*)\right]\\
  &\quad +\frac{21}{8}(1 - \beta)^2\|(\theta_k - \vartheta_k)^\top\phi_k\|^2 - \frac{21}{32}\|(\vartheta_k - \theta^*)^\top\phi_k\|^2\Bigg\}\\
  &\leq -\frac{11}{8}\|\theta_k - \vartheta_k\|^2 + \frac{1}{\N_k}\Bigg\{-\left[\|\tilde{\theta}_{k+1}^\top\phi_k\| - 2\|\phi_k\|\|\theta_k - \vartheta_k\|\right]^2 - 4\|\theta_k - \vartheta_k\|^2 \\
  &\quad- \frac{7}{8}\left[2\left\|(1 - \beta)(\theta_k - \vartheta_k)^\top\phi_k\right\| - \frac{1}{2}\|(\vartheta_k - \theta^*)^\top\phi_k\|\right]^2\\
  &\quad -\frac{21}{32}\|(\vartheta_k - \theta^*)^\top\phi_k\|^2\Bigg\}\\
  &\leq 0
\end{align*}
  proving Theorem \ref{theo:1} (i).

  Summing $V_k$ from $V_{k-\Delta T}$ to $V_{k}$, we have
  \begin{align}
    &V_k - V_{k - \Delta T}\nonumber\\
    &= V_k - V_{k-1} + \cdots + V_{k-\Delta T + 1} - V_{k-\Delta T}\nonumber\\
    &\leq -c_1\!\sum_{i=k-\Delta T}^{k-1}\|\theta_i - \vartheta_i\|^2 - c_2\!\sum_{i=k-\Delta T}^{k-1} \frac{1}{\N_i}\|(\vartheta_i - \theta^*)^\top\phi_i\|^2\nonumber\\
    &\leq -c_1\frac{1}{\Delta T}\Bigg[\underbrace{\sum_{i=k-\Delta T}^{k-1}\left\|\theta_i - \vartheta_i\right\|}_{W_1}\Bigg]^2 - c_2\frac{1}{\Delta T}\Bigg[\underbrace{\sum_{i=k - \Delta T}^{k-1}\frac{1}{\sqrt{\N_i}}\left\|(\vartheta_i - \theta^*)^\top\phi_i\right\|}_{W_2}\Bigg]^2
      \label{eq:upper}
  \end{align}
where we applied the Cauchy-Schwarz inequality for the second inequality. We consider two cases: $\|\vartheta_{k-\Delta T} - \theta^*\|^2 \geq \lambda\gamma V_{k-\Delta T}$ and $\|\vartheta_{k-\Delta T} - \theta^*\|^2 \leq \lambda\gamma V_{k-\Delta T}$, where $\lambda$ satisfies \eqref{eq:lambda}.\\
  \noindent\textbf{Case 1}: $\|\vartheta_{k-\Delta T} - \theta^*\|^2 \geq \lambda\gamma V_{k-\Delta T}$
  
    \noindent\textbf{(a)} If $W_1 \geq \eta \|\vartheta_{k-\Delta T} - \theta^*\|$, where $\eta$ satisfies \eqref{eq:eta}, then
    \begin{align*}
      V_k - V_{k-\Delta T} &\leq -c_1\frac{1}{\Delta T}\eta^2 \|\vartheta_{k-\Delta T} - \theta^*\|^2\\
                           &\leq -c_1\frac{1}{\Delta T}\lambda\gamma\eta^2 V_{k-\Delta T}
    \end{align*}
    and
    \begin{equation}
      V_k \leq \left(1 - c_1\frac{1}{\Delta T}\lambda\gamma\eta^2\right) V_{k-\Delta T}.
      \label{eq:v11}
    \end{equation}
    
  \noindent\textbf{(b)} If $W_1 \leq \eta \|\vartheta_{k-\Delta T} - \theta^*\|$, then
  \begin{align*}
    W_2 &= \sum_{i = k - \Delta T}^{k-1}\frac{1}{\sqrt{\N_i}}\|(\vartheta_i - \vartheta_{k - \Delta T} + \vartheta_{k - \Delta T} - \theta^*)^\top\phi_i\|\\
    &\geq \sum_{i = k - \Delta T}^{k-1}\frac{1}{\sqrt{\N_i}}\|(\vartheta_{k-\Delta T} - \theta^*)^\top\phi_i\| - \sum_{i = k - \Delta T}^{k-1}\frac{1}{\sqrt{\N_i}}\|(\vartheta_i - \vartheta_{k - \Delta T})^\top\phi_i\|\\
    &\geq \Delta T \epsilon_1\|\vartheta_{k-\Delta T} - \theta^*\| - \Delta T \! \sup_{i\in[k - \Delta T, k-1]}\!\|\vartheta_i - \vartheta_{k - \Delta T}\|\\
    &\geq \Delta T \epsilon_1\|\vartheta_{k-\Delta T} - \theta^*\| - \Delta T \sum_{i=k - \Delta T}^{k-2}\|\vartheta_{i+1} - \vartheta_i\|\\
    &\geq \Delta T \epsilon_1\|\vartheta_{k-\Delta T} - \theta^*\| - \gamma \Delta T \sum_{i=k - \Delta T}^{k - 1}\left\|\frac{\phi_i\phi_i^\top\tilde{\theta}_{i+1}}{\N_i}\right\|\\
    &\geq \Delta T \epsilon_1\|\vartheta_{k-\Delta T} - \theta^*\| - \gamma\Delta T \sum_{i = k - \Delta T}^{k - 1}\frac{1}{\sqrt{\N_i}}\|(\theta_{i+1} - \vartheta_i + \vartheta_i - \theta^*)^\top\phi_i\|\\
    &\geq \Delta T \epsilon_1\|\vartheta_{k-\Delta T} - \theta^*\| - \gamma\Delta T \!\sum_{i=k-\Delta T}^{k-1}\!\frac{1}{\sqrt{\N_i}}\|(\vartheta_i - \theta^*)^\top\phi_i\| - \gamma\Delta T\sum_{i=k-\Delta T}^{k-1}\frac{|1-\beta|}{\sqrt{\N_i}}\|(\theta_i - \vartheta_i)^\top\phi_i\|\\
    &\geq \Delta T \epsilon_1\|\vartheta_{k-\Delta T} - \theta^*\| - \gamma\Delta T\! \underbrace{\sum_{i=k-\Delta T}^{k-1}\!\frac{1}{\sqrt{\N_i}}\|(\vartheta_i - \theta^*)^\top\phi_i\|}_{W_2} - \gamma|1-\beta|\Delta T\underbrace{\sum_{i=k-\Delta T}^{k-1}\|\theta_i - \vartheta_i\|}_{W_1}
  \end{align*}
From \eqref{eq:eta}, $\epsilon_1 > \gamma\eta|1 - \beta|$. Therefore
\begin{align}
  W_2 \geq \frac{\Delta T}{1 + \gamma\Delta T}(\epsilon_1 - \gamma\eta|1-\beta|)\|\vartheta_{k-\Delta T} - \theta^*\| \geq 0.
  \label{eq:w21}
\end{align}
Substitute \eqref{eq:w21} into \eqref{eq:upper}, we obtain
\begin{align*}
  V_k - V_{k-\Delta T} &\leq -c_2\frac{1}{\Delta T} W_2^2\nonumber\\
                       &\leq -c_2\frac{\Delta T}{(1 + \gamma\Delta T)^2}(\epsilon_1 - \gamma\eta|1 - \beta|)^2\|\vartheta_{k-\Delta T} - \theta^*\|^2\nonumber\\
                       &\leq -c_2 \frac{\Delta T}{(1 + \gamma\Delta T)^2}(\epsilon_1 - \gamma\eta|1 - \beta|)^2\lambda\gamma V_{k-\Delta T}
\end{align*}
And thus
\begin{equation}
  \label{eq:v12}
  V_k \leq \left[1 -c_2 \frac{\Delta T}{(1 + \gamma\Delta T)^2}(\epsilon_1 - \gamma\eta|1 - \beta|)^2\lambda\gamma\right] V_{k-\Delta T}.
\end{equation}

\noindent\textbf{Case 2}: $\|\vartheta_{k-\Delta T}-\theta^*\|^2\leq \lambda\gamma V_{k-\Delta T}$

From
\begin{equation*}
  V_{k-\Delta T} = \frac{1}{\gamma}\|\theta_{k-\Delta T} - \vartheta_{k-\Delta T}\|^2 + \frac{1}{\gamma}\|\vartheta_{k-\Delta T} - \theta^*\|^2,
\end{equation*}
we immediately get
\begin{equation}
  \|\theta_{k-\Delta T} - \vartheta_{k-\Delta T}\|^2 \geq (1 - \lambda)\gamma V_{k-\Delta T}.
  \label{eq:41}
\end{equation}
Since $W_1 \geq \|\theta_{k-\Delta T} - \vartheta_{k-\Delta T}\|$, from \eqref{eq:upper},
\begin{align*}
  V_k - V_{k-\Delta T} &\leq -c_1\frac{1}{\Delta T} W_1^2\\
                       &\leq -c_1\frac{1}{\Delta T} \|\theta_{k-\Delta T} - \vartheta_{k-\Delta T}\|^2\\
                       &\leq -c_1\frac{1}{\Delta T} (1-\lambda)\gamma V_{k-\Delta T}
\end{align*}
From which we obtain
\begin{equation}
  \label{eq:v13}
  V_k \leq \left[1 - c_1\frac{1}{\Delta T}(1 - \lambda)\gamma\right] V_{k-\Delta T}
\end{equation}
Considering \eqref{eq:v11}, \eqref{eq:v12} and \eqref{eq:v13}, we have
\begin{equation*}
  V_k \leq (1 - \mu) V_{k-\Delta T},
\end{equation*}
where $\mu$ is defined in \eqref{eq:mu}. Collecting the terms, we obtain
\begin{equation*}
V_k \leq \exp\left(-\mu\left\lfloor\frac{k}{\Delta T}\right\rfloor\right)V_0,
\end{equation*}
proving Theorem \ref{theo:1} (ii).
\end{proof}
\begin{remark}
  It should be noted that the proof of Theorem \ref{theo:1} followed by considering two different cases. In case 1, we assumed that the parameter error in $\vartheta$ was a significant fraction of the overall Lyapunov function. We showed then that a decrease in $V_k$ is either due to the nature of the high order tuner, or persistent excitation of $\phi_k$. In case 2, the parameter difference between $\theta$ and $\vartheta$ was a significant fraction of the Lyapunov function, which directly leads to a decrease in $V_k$ due to the nature of the HB algorithm in \eqref{eq:HT-HB}.
\end{remark}

\section{Main Result 2: Accelerated Learning with Nesterov's Acceleration}
\label{sec:nesterov}
The idea behind the second HT is motivated by \cite{Nesterov_1983} and an important stability-preserving variation of the same proposed in \cite{Gaudio20AC}. Similar to \eqref{eq:HT-HB}, this high-order tuner uses not just the past iterate $\theta_k$ to determine $\theta_{k+1}$, but also $\theta_{k-1}$. However, in addition to the momentum term, an acceleration-based addition is included as well. A simplified version of the Nesterov's algorithm in \cite{Nesterov_1983} is of the form
\begin{equation}
  \label{eq:ht1}
  \theta_{k+1}=\theta_k-\bar\gamma\frac{\nabla L_k(\theta_{k}+\bar\beta (\theta_k-\theta_{k-1}))}{\N_k} + \bar\beta (\theta_k-\theta_{k-1}),
\end{equation}
where $\bar\gamma$ and $\bar\beta$ are hyperparameters and the second term corresponds to the momentum addition, as it computes the gradient based on an updated parameter estimate. It was shown in \cite{Gaudio20AC} that such an update cannot be shown to be stable when adversarial regressors, which may be time-varying, are present. An important modification of the same was introduced, which can be expressed in the form of two first-order iterates
\begin{equation}
  \label{eq:ht}
  \begin{split}
    \vartheta_{k+1} &= \vartheta_k - \gamma\frac{\nabla L_{k}(\theta_{k+1})}{\N_k},\\
    \theta_{k+1} &= \bar\theta_k - \beta(\bar\theta_k - \vartheta_k),\\
    \bar\theta_k &= \theta_k - \gamma\beta\frac{\nabla L_k(\theta_k)}{\N_k},
  \end{split}
\end{equation}
where $\beta$ and $\gamma$ are chosen such that $0 < \beta < 1$ and $0 < \gamma \leq \frac{\beta(2 - \beta)}{16 + \beta^2}$. We denote this as the NA algorithm. The update law in \eqref{eq:ht} was shown in \cite{Gaudio20AC} to be stable using \eqref{eq:lyap} as the Lyapunov function.  Before stating the main result, we define a few parameters.

Let
\begin{equation*}
  c_3 = \frac{7}{4}, \quad c_4 = \frac{9}{16},
\end{equation*}
\begin{equation*}
  \epsilon_2 = \frac{\epsilon}{\max_k\{\sqrt{\N_k}\}},
\end{equation*}
\begin{equation}
0 < \lambda < 1,
\label{eq:lambda2}
\end{equation}
\begin{equation}
0 < \eta < \frac{\epsilon_2}{\gamma(1 - \beta)},
\label{eq:eta2}
\end{equation}
\begin{equation}
0 < \zeta < \frac{1 - \gamma\beta}{\gamma(1 + \beta - \gamma\beta)},
\label{eq:zeta2}
\end{equation}
and define
\begin{equation}
  \label{eq:mu2}
  \mu = \min\{\mu_1, \mu_2, \mu_3, \mu_4\},
\end{equation}
where
\begin{align*}
  \mu_1 &= \frac{c_3\lambda\gamma\eta^2}{\Delta T},\\
  \mu_2 &= \frac{c_4\Delta T[\epsilon_2 - \gamma\eta(1 - \beta)]^2\lambda\gamma}{(1 + \gamma\Delta T)^2},\\
  \mu_3 &= \frac{c_4\gamma\zeta^2(1 - \lambda)}{\Delta T},\\
  \mu_4 &= \frac{c_3\xi^2(1 - \lambda)\gamma}{\Delta T},
\end{align*}
and
\begin{equation*}
  \xi = \frac{\Delta T - \gamma\zeta\Delta T - \frac{\zeta\gamma\beta(1 + \Delta T)}{1 - \gamma\beta}}{1 + \gamma(1 - \beta)\Delta T + \beta\Delta T + \zeta\gamma\beta\frac{1 + \Delta T}{1 - \gamma\beta}}.
\end{equation*}

We now state our second main result:
\begin{theorem}
  If the regressor $\phi_k$ satisfies the PE definition in \eqref{eq:pe}, with $0 < \beta < 1$ and $0 < \gamma \leq \frac{\beta(2 - \beta)}{8 + \beta^2}$, the update law in \eqref{eq:ht} will result in (i) $\vartheta_k - \theta^* \in
 \ell_\infty$, $\theta_k - \vartheta_k\in \ell_\infty$, and (ii) $V_k \leq \exp\left(-\mu\left\lfloor\frac{k}{\Delta T}\right\rfloor\right)V_0$, where $\mu$ is defined in \eqref{eq:mu2}.
  \label{theo:2}
\end{theorem}

\begin{proof}
  Expanding $\Delta V_k := V_{k +1} - V_k$, we have
\begin{align*}
  \Delta V_k
  &=\frac{1}{\gamma}\lVert \vartheta_{k+1}-\theta^*\rVert^2+\frac{1}{\gamma}\lVert \theta_{k+1}-\vartheta_{k+1}\rVert^2-\frac{1}{\gamma}\lVert \vartheta_k-\theta^*\rVert^2-\frac{1}{\gamma}\lVert \theta_k-\vartheta_k\rVert^2\\
  &=\frac{1}{\gamma}\lVert (\vartheta_k-\theta^*)-\frac{\gamma}{\N_k}\nabla L_k(\theta_{k+1})\rVert^2-\frac{1}{\gamma}\lVert \vartheta_k-\theta^*\rVert^2\\
  &\quad +\frac{1}{\gamma}\lVert \bar{\theta}_k-\beta(\bar{\theta}_k-\vartheta_k)-\vartheta_k+\frac{\gamma}{\N_k}\nabla L_k(\theta_{k+1})\rVert^2-\frac{1}{\gamma}\lVert \theta_k-\vartheta_k\rVert^2\\
  &=\frac{\gamma}{\N_k^2}\lVert\nabla L_k(\theta_{k+1})\rVert^2-\frac{2}{\N_k}(\vartheta_k-\theta^*)^\top\nabla L_k(\theta_{k+1})\\
  &\quad+\frac{1}{\gamma}\lVert \bar{\theta}_k-\vartheta_k\rVert^2-\frac{1}{\gamma}\lVert \theta_k-\vartheta_k\rVert^2-\frac{\beta(2-\beta)}{\gamma}\lVert \bar{\theta}_k-\vartheta_k\rVert^2\\
  &\quad +\frac{2}{\N_k}(1-\beta)(\bar{\theta}_k-\vartheta_k)^\top\nabla L_k(\theta_{k+1})+\frac{\gamma}{\N_k^2}\lVert\nabla L_k(\theta_{k+1})\rVert^2\\
  &=\frac{2\gamma}{\N_k^2}\lVert\nabla L_k(\theta_{k+1})\rVert^2-\frac{2}{\N_k}(\theta_{k+1}-\theta^*)^\top\nabla L_k(\theta_{k+1})\\
  &\quad+\frac{1}{\gamma}\lVert \bar{\theta}_k-\vartheta_k\rVert^2-\frac{1}{\gamma}\lVert \theta_k-\vartheta_k\rVert^2-\frac{\beta(2-\beta)}{\gamma}\lVert \bar{\theta}_k-\vartheta_k\rVert^2\\
  &\quad +\frac{2}{\N_k}(1-\beta)(\bar{\theta}_k-\vartheta_k)^\top\nabla L_k(\theta_{k+1})-\frac{2}{\N_k}(\vartheta_k-\theta_{k+1})^\top\nabla L_k(\theta_{k+1})\\
  &=-2\left(1-\frac{\gamma\phi_k^\top\phi_k}{\N_k}\right)\frac{\tilde{\theta}_{k+1}^\top\nabla L_k(\theta_{k+1})}{\N_k}\\
  &\quad+\frac{1}{\gamma}\lVert \bar{\theta}_k-\vartheta_k\rVert^2-\frac{1}{\gamma}\lVert \theta_k-\vartheta_k\rVert^2-\frac{\beta(2-\beta)}{\gamma}\lVert \bar{\theta}_k-\vartheta_k\rVert^2\\
  &\quad +\frac{4}{\N_k}(1-\beta)(\bar{\theta}_k-\vartheta_k)^\top\nabla L_k(\theta_{k+1})\\
  &=-2\left(1-\frac{\gamma\phi_k^\top\phi_k}{\N_k}\right)\frac{\tilde{\theta}_{k+1}^\top\nabla L_k(\theta_{k+1})}{\N_k}\\
  &\quad+\frac{\gamma\beta^2}{\N_k^2}\lVert\nabla L_k(\theta_k)\rVert^2-\frac{2\beta}{\N_k}(\theta_k-\vartheta_k)^\top\nabla L_k(\theta_k)\\
  &\quad -\frac{\beta(2-\beta)}{\gamma}\lVert \bar{\theta}_k-\vartheta_k\rVert^2+\frac{4}{\N_k}(1-\beta)(\bar{\theta}_k-\vartheta_k)^\top\nabla L_k(\theta_{k+1})\\
  &=-\frac{\beta(2-\beta)}{\gamma}\lVert \bar{\theta}_k-\vartheta_k\rVert^2 + \frac{1}{\N_k}\bigg\{-2\left(1-\frac{\gamma\phi_k^\top\phi_k}{\N_k}\right)\tilde{\theta}_{k+1}^\top\nabla L_k(\theta_{k+1})+\frac{\gamma\beta^2}{\N_k}\lVert\nabla L_k(\theta_k)\rVert^2\\
  &\quad -2\beta(\theta_k-\vartheta_k)^\top\nabla L_k(\theta_k) +4(1-\beta)(\bar{\theta}_k-\vartheta_k)^\top\nabla L_k(\theta_{k+1})\bigg\}\\
  &=-\frac{\beta(2-\beta)}{\gamma}\lVert \bar{\theta}_k-\vartheta_k\rVert^2 \\
  &\quad + \frac{1}{\N_k}\bigg\{-2\left(1-\frac{\gamma\phi_k^\top\phi_k}{\N_k}\right)\tilde{\theta}_{k+1}^\top\nabla L_k(\theta_{k+1})-\frac{\gamma\beta^2}{\N_k}\lVert\nabla L_k(\theta_k)\rVert^2-2\beta(\bar{\theta}_k-\vartheta_k)^\top\nabla L_k(\theta_k)\\
  &\quad +4(1-\beta)(\bar{\theta}_k-\vartheta_k)^\top\nabla L_k(\theta_{k+1})\bigg\}\\
  &=-\frac{\beta(2-\beta)}{\gamma}\lVert \bar{\theta}_k-\vartheta_k\rVert^2\\
  &\quad+\frac{1}{\N_k}\bigg\{-2\left(1-\frac{\gamma\phi_k^\top\phi_k}{\N_k}\right)\tilde{\theta}_{k+1}^\top\nabla L_k(\theta_{k+1})-\frac{\gamma\beta^2}{\N_k}\lVert\nabla L_k(\theta_k)\rVert^2-2\beta(\bar{\theta}_k-\vartheta_k)^\top\phi_k\phi^\top_k\tilde{\theta}_k \\
  &\quad +4(1-\beta)(\bar{\theta}_k-\vartheta_k)^\top\nabla L_k(\theta_{k+1})\bigg\}\\
  &=-\frac{\beta(2-\beta)}{\gamma}\lVert \bar{\theta}_k-\vartheta_k\rVert^2\\
  &\quad +\frac{1}{\N_k}\bigg\{-2\left(1-\frac{\gamma\phi_k^\top\phi_k}{\N_k}\right)\tilde{\theta}_{k+1}^\top\nabla L_k(\theta_{k+1})-\frac{\gamma\beta^2}{\N_k}\lVert\nabla L_k(\theta_k)\rVert^2\\
  &\quad+4(1-\beta)(\bar{\theta}_k-\vartheta_k)^\top\nabla L_k(\theta_{k+1})\\ &\quad-2\beta(\bar{\theta}_k-\vartheta_k)^\top\phi_k\phi^\top_k\left[\theta_k-\theta^*+(1-\beta)\bar{\theta}_k+ \beta\vartheta_k-(1-\beta)\bar{\theta}_k-\beta\vartheta_k\right]\bigg\}\\
  &=-\frac{\beta(2-\beta)}{\gamma}\lVert \bar{\theta}_k-\vartheta_k\rVert^2\\
  &\quad + \frac{1}{\N_k}\bigg\{-2\left(1-\frac{\gamma\phi_k^\top\phi_k}{\N_k}\right)\tilde{\theta}_{k+1}^\top\nabla L_k(\theta_{k+1})-\frac{\gamma\beta^2}{\N_k}\lVert\nabla L_k(\theta_k)\rVert^2\\
  &\quad-2\beta(\bar{\theta}_k-\vartheta_k)^\top\nabla L_k(\theta_{k+1})\\
  &\quad+4(1-\beta)(\bar{\theta}_k-\vartheta_k)^\top\nabla L_k(\theta_{k+1})-2\beta(\bar{\theta}_k-\vartheta_k)^\top\phi_k\phi^\top_k\left[\theta_k-\bar{\theta}_k+\beta(\bar{\theta}_k-\vartheta_k)\right]\bigg\}\\
  &=-\frac{\beta(2-\beta)}{\gamma}\lVert \bar{\theta}_k-\vartheta_k\rVert^2\\
  &\quad+\frac{1}{\N_k}\bigg\{-2\left(1-\frac{\gamma\phi_k^\top\phi_k}{\N_k}\right)\tilde{\theta}_{k+1}^\top\nabla L_k(\theta_{k+1})-\frac{\gamma\beta^2}{\N_k}\lVert\nabla L_k(\theta_k)\rVert^2 \\
  &\quad -2\gamma\beta^2(\bar{\theta}_k-\vartheta_k)^\top\frac{\phi_k\phi^\top_k}{\N_k}\nabla L_k(\theta_k)  \\
  &\quad +4(1-\frac{3}{2}\beta)(\bar{\theta}_k-\vartheta_k)^\top\nabla L_k(\theta_{k+1})-2\beta^2(\bar{\theta}_k-\vartheta_k)^\top\phi_k\phi^\top_k(\bar{\theta}_k-\vartheta_k)\bigg\}\\
  &\leq -4\|\bar\theta_k - \vartheta_k\|^2\\
  &\quad+\frac{1}{\N_k}\bigg\{-2\left(1-\frac{\gamma\phi_k^\top\phi_k}{\N_k}\right)\tilde{\theta}_{k+1}^\top\nabla L_k(\theta_{k+1}) - 4\lVert\phi_k\rVert^2\lVert \bar{\theta}_k-\vartheta_k\rVert^2\\
  &\quad+4(1-\frac{3}{2}\beta)(\bar{\theta}_k-\vartheta_k)^\top\nabla L_k(\theta_{k+1})\\
  &\quad-\frac{\gamma\beta^2}{\N_k}\lVert\nabla L_k(\theta_k)\rVert^2-\beta^2\lVert\phi_k\rVert^2\lVert \bar{\theta}_k-\vartheta_k\rVert^2-2\beta^2(\bar{\theta}_k-\vartheta_k)^\top\gamma\frac{\phi_k\phi^\top_k}{\N_k}\nabla L_k(\theta_k)\\
  &\quad-(4+\beta^2)\lVert \bar{\theta}_k-\vartheta_k\rVert^2-2\beta^2(\bar{\theta}_k-\vartheta_k)^\top\phi_k\phi^\top_k(\bar{\theta}_k-\vartheta_k)\bigg\}\\
  &\leq -4\|\bar\theta_k - \vartheta_k\|^2\\
  &\quad+\frac{1}{\N_k}\bigg\{-\frac{7}{4}\lVert\tilde{\theta}_{k+1}^\top\phi_k\rVert^2-4\lVert\phi_k\rVert^2\lVert \bar{\theta}_k-\vartheta_k\rVert^2+4\lVert\bar{\theta}_k-\vartheta_k\rVert\lVert\phi_k\rVert\lVert\tilde{\theta}_{k+1}^\top\phi_k\rVert\\
  &\quad -\frac{\gamma\beta^2}{\N_k}\lVert\nabla L_k(\theta_k)\rVert^2-\beta^2\lVert\phi_k\rVert^2\lVert \bar{\theta}_k-\vartheta_k\rVert^2+2\beta^2\lVert\bar{\theta}_k-\vartheta_k\rVert\frac{\lVert\sqrt{\gamma}\phi_k\rVert^2}{\N_k}\lVert\nabla L_k(\theta_k)\rVert\\
  &\quad -(4+\beta^2)\lVert \bar{\theta}_k-\vartheta_k\rVert^2-2\beta^2(\bar{\theta}_k-\vartheta_k)^\top\phi_k\phi^\top_k(\bar{\theta}_k-\vartheta_k)\bigg\}\\
  &\leq -4\|\bar\theta_k - \vartheta_k\|^2\\
    &\quad+ \frac{1}{\N_k}\bigg\{-\lVert\tilde{\theta}_{k+1}^\top\phi_k\rVert^2-4\lVert\phi_k\rVert^2\lVert \bar{\theta}_k-\vartheta_k\rVert^2+4\lVert\bar{\theta}_k-\vartheta_k\rVert\lVert\phi_k\rVert\lVert\tilde{\theta}_{k+1}^\top\phi_k\rVert \\
  &\quad -\frac{\gamma\beta^2}{\N_k}\lVert\nabla L_k(\theta_k)\rVert^2-\beta^2\lVert\phi_k\rVert^2\lVert \bar{\theta}_k-\vartheta_k\rVert^2+\frac{2\sqrt{\gamma}\beta^2}{\sqrt{\N_k}}\lVert\bar{\theta}_k-\vartheta_k\rVert\lVert\phi_k\rVert\lVert\nabla L_k(\theta_k)\rVert \\
  &\quad -\frac{3}{4}\lVert\tilde{\theta}_{k+1}^\top\phi_k\rVert^2-(4+\beta^2)\lVert \bar{\theta}_k-\vartheta_k\rVert^2-2\beta^2(\bar{\theta}_k-\vartheta_k)^\top\phi_k\phi^\top_k(\bar{\theta}_k-\vartheta_k)\bigg\}\\
  &\leq-4\|\bar\theta_k - \vartheta_k\|^2\\
  &\quad + \frac{1}{\N_k}\bigg\{-\left[\lVert\tilde{\theta}_{k+1}^\top\phi_k\rVert-2\lVert\phi_k\rVert\lVert \bar{\theta}_k-\vartheta_k\rVert\right]^2 \\
  &\quad -\left[\frac{\sqrt{\gamma}\beta}{\sqrt{\N_k}}\lVert\nabla L_k(\theta_k)\rVert-\beta\lVert\phi_k\rVert\lVert \bar{\theta}_k-\vartheta_k\rVert\right]^2 \\
  &\quad -\frac{3}{4}\lVert\tilde{\theta}_{k+1}^\top\phi_k\rVert^2-(4+\beta^2)\lVert \bar{\theta}_k-\vartheta_k\rVert^2-2\beta^2\lVert(\bar{\theta}_k-\vartheta_k)^\top\phi_k\rVert^2\bigg\}\\
  &\leq-4\|\bar\theta_k - \vartheta_k\|^2\\
  &\quad + \frac{1}{\N_k}\bigg\{-\left[\lVert\tilde{\theta}_{k+1}^\top\phi_k\rVert-2\lVert\phi_k\rVert\lVert \bar{\theta}_k-\vartheta_k\rVert\right]^2 \\
  &\quad -\left[\frac{\sqrt{\gamma}\beta}{\sqrt{\N_k}}\lVert\nabla L_k(\theta_k)\rVert-\beta\lVert\phi_k\rVert\lVert \bar{\theta}_k-\vartheta_k\rVert\right]^2 \\
  &\quad -\frac{3}{4}\lVert(\theta_{k+1} - \vartheta_k + \vartheta_k - \theta^*)^\top \phi_k \rVert^2-(4+\beta^2)\lVert \bar{\theta}_k-\vartheta_k\rVert^2-2\beta^2\lVert(\bar{\theta}_k-\vartheta_k)^\top\phi_k\rVert^2\bigg\}\\
  &\leq-4\|\bar\theta_k - \vartheta_k\|^2\\
  &\quad + \frac{1}{\N_k}\bigg\{-\left[\lVert\tilde{\theta}_{k+1}^\top\phi_k\rVert-2\lVert\phi_k\rVert\lVert \bar{\theta}_k-\vartheta_k\rVert\right]^2 \\
  &\quad -\left[\frac{\sqrt{\gamma}\beta}{\sqrt{\N_k}}\lVert\nabla L_k(\theta_k)\rVert-\beta\lVert\phi_k\rVert\lVert \bar{\theta}_k-\vartheta_k\rVert\right]^2 \\
  &\quad -\frac{3}{4}\lVert(1-\beta)(\bar\theta_k - \vartheta_k)^\top\phi_k + (\vartheta_k - \theta^*)^\top\phi_k\rVert^2-(4+\beta^2)\lVert \bar{\theta}_k-\vartheta_k\rVert^2-2\beta^2\lVert(\bar{\theta}_k-\vartheta_k)^\top\phi_k\rVert^2\bigg\}\\
  &\leq-4\|\bar\theta_k - \vartheta_k\|^2\\
  &\quad + \frac{1}{\N_k}\bigg\{-\left[\lVert\tilde{\theta}_{k+1}^\top\phi_k\rVert-2\lVert\phi_k\rVert\lVert \bar{\theta}_k-\vartheta_k\rVert\right]^2 \\
  &\quad -\left[\frac{\sqrt{\gamma}\beta}{\sqrt{\N_k}}\lVert\nabla L_k(\theta_k)\rVert-\beta\lVert\phi_k\rVert\lVert \bar{\theta}_k-\vartheta_k\rVert\right]^2 \\
  &\quad -(4+\beta^2)\lVert \bar{\theta}_k-\vartheta_k\rVert^2-2\beta^2\lVert(\bar{\theta}_k-\vartheta_k)^\top\phi_k\rVert^2 \\
  &\quad -\frac{3}{4}\left[\left\|(1 - \beta)(\bar\theta_k - \vartheta_k)^\top\phi_k\right\|^2 + \left\|(\vartheta_k - \theta^*)^\top\phi_k\right\|^2 + 2(1 - \beta)(\bar\theta_k - \vartheta_k)^\top\phi_k\phi_k^\top (\vartheta_k - \theta^*)\right]\bigg\}\\
  &\leq -4\|\bar\theta_k - \vartheta_k\|^2\\
  &\quad + \frac{1}{\N_k}\bigg\{-\left[\lVert\tilde{\theta}_{k+1}^\top\phi_k\rVert-2\lVert\phi_k\rVert\lVert \bar{\theta}_k-\vartheta_k\rVert\right]^2 \\
  &\quad -\left[\frac{\sqrt{\gamma}\beta}{\sqrt{\N_k}}\lVert\nabla L_k(\theta_k)\rVert-\beta\lVert\phi_k\rVert\lVert \bar{\theta}_k-\vartheta_k\rVert\right]^2 \\
  &\quad -(4+\beta^2)\lVert \bar{\theta}_k-\vartheta_k\rVert^2-2\beta^2\lVert(\bar{\theta}_k-\vartheta_k)^\top\phi_k\rVert^2 \\
  &\quad -\frac{3}{4}\left[4\left\|(1 - \beta)(\bar\theta_k - \vartheta_k)^\top\phi_k\right\|^2 + \frac{1}{4} \left\|(\vartheta_k - \theta^*)^\top\phi_k\right\|^2 + 2(1 - \beta)(\bar\theta_k - \vartheta_k)^\top\phi_k\phi_k^\top (\vartheta_k - \theta^*)\right] \\
  &\quad +\frac{9}{4}\left\|(1 - \beta)(\bar\theta_k - \vartheta_k)^\top\phi_k\right\|^2 - \frac{9}{16} \left\|(\vartheta_k - \theta^*)^\top\phi_k\right\|^2\bigg\}\\
  &\leq-\frac{7}{4}\|\bar\theta_k - \vartheta_k\|^2\\
  &\quad + \frac{1}{\N_k}\bigg\{-\left[\lVert\tilde{\theta}_{k+1}^\top\phi_k\rVert-2\lVert\phi_k\rVert\lVert \bar{\theta}_k-\vartheta_k\rVert\right]^2 - \left[\frac{\sqrt{\gamma}\beta}{\sqrt{\N_k}}\lVert\nabla L_k(\theta_k)\rVert-\beta\lVert\phi_k\rVert\lVert \bar{\theta}_k-\vartheta_k\rVert\right]^2 \\
  &\quad -\left(\frac{25}{4}+\beta^2\right)\lVert \bar{\theta}_k-\vartheta_k\rVert^2-2\beta^2\lVert(\bar{\theta}_k-\vartheta_k)^\top\phi_k\rVert^2 \\
  &\quad -\frac{3}{4}\left[2(1-\beta)\left\|(\theta_k - \vartheta_k)^\top\phi_k\right\| - \frac{1}{2}\left\|(\vartheta_k - \theta^*)^\top\phi_k\right\|\right]^2 \\
  &\quad - \frac{9}{16} \left\|(\vartheta_k - \theta^*)^\top\phi_k\right\|^2\bigg\}\\
  &\leq 0
\end{align*}
  proving Theorem \ref{theo:2} (i).

  Summing $V_k$ from $V_{k-\Delta T}$ to $V_k$, we obtain
\begin{align}
  &V_k - V_{k - \Delta T}\nonumber\\
  &= V_k - V_{k-1} + \cdots + V_{k-\Delta T + 1} - V_{k-\Delta T}\nonumber\\
  &\leq -c_3\!\sum_{i=k-\Delta T}^{k-1}\|\bar\theta_i - \vartheta_i\|^2 - c_4\!\sum_{i=k-\Delta T}^{k-1} \frac{1}{\N_i}\|(\vartheta_i - \theta^*)^\top\phi_i\|^2\nonumber\\
  &\leq -c_3\frac{1}{\Delta T}\Bigg(\underbrace{\sum_{i=k-\Delta T}^{k-1}\|\bar\theta_i - \vartheta_i\|}_{W_1}\Bigg)^2 -c_4\frac{1}{\Delta T}\Bigg(\underbrace{\sum_{i=k-\Delta T}^{k-1} \frac{1}{\sqrt{\N_i}}\|(\vartheta_i - \theta^*)^\top\phi_i\|}_{W_2}\Bigg)^2,
    \label{eq:ub}
\end{align}
where we applied the Cauchy-Schwarz inequality for the second inequality.

\noindent\textbf{Case 1}: $\|\vartheta_{k-\Delta T}-\theta^*\|^2\geq \lambda\gamma V_{k-\Delta T}$, where $\lambda$ satisfies \eqref{eq:lambda2}.

\noindent\textbf{(a)} If $W_1 \geq \eta \|\vartheta_{k-\Delta T} - \theta^*\|$, where $\eta$ satisfies \eqref{eq:eta2}, then
\begin{align*}
  V_k - V_{k-\Delta T} &\leq -c_3\frac{1}{\Delta T}\eta^2 \|\vartheta_{k-\Delta T} - \theta^*\|^2\\
                       &\leq -c_3\frac{1}{\Delta T}\lambda\gamma\eta^2 V_{k-\Delta T}
\end{align*}
and
\begin{align}
  V_k \leq \left(1 - c_3\frac{1}{\Delta T}\lambda\gamma\eta^2\right) V_{k-\Delta T}
  \label{eq:v21}
\end{align}

\noindent\textbf{(b)} If $W_1 \leq \eta \|\vartheta_{k-\Delta T} - \theta^*\|$, then
\begin{align*}
  W_2 &= \sum_{i = k - \Delta T}^{k-1}\frac{1}{\sqrt{\N_i}}\|(\vartheta_i - \vartheta_{k - \Delta T} + \vartheta_{k - \Delta T} - \theta^*)^\top\phi_i\|\\
  &\geq \sum_{i = k - \Delta T}^{k-1}\frac{1}{\sqrt{\N_i}}\|(\vartheta_{k-\Delta T} - \theta^*)^\top\phi_i\| - \sum_{i = k - \Delta T}^{k-1}\frac{1}{\sqrt{\N_i}}\|(\vartheta_i - \vartheta_{k - \Delta T})^\top\phi_i\|\\
  &\geq \Delta T \epsilon_2\|\vartheta_{k-\Delta T} - \theta^*\| - \Delta T \sup_{i\in[k - \Delta T, k-1]}\|\vartheta_i - \vartheta_{k - \Delta T}\|\\
  &\geq \Delta T \epsilon_2\|\vartheta_{k-\Delta T} - \theta^*\| - \Delta T \sum_{i=k - \Delta T}^{k-2}\|\vartheta_{i+1} - \vartheta_i\|\\
  &\geq \Delta T \epsilon_2\|\vartheta_{k-\Delta T} - \theta^*\| - \gamma \Delta T \sum_{i=k - \Delta T}^{k - 1}\left\|\frac{\phi_i\phi_i^\top\tilde{\theta}_{i+1}}{\N_i}\right\|\\
  &\geq \Delta T \epsilon_2\|\vartheta_{k-\Delta T} - \theta^*\| - \gamma\Delta T \sum_{i = k - \Delta T}^{k - 1}\frac{1}{\sqrt{\N_i}}\|(\theta_{i+1} - \vartheta_i + \vartheta_i - \theta^*)^\top\phi_i\|\\
  &\geq \Delta T \epsilon_2\|\vartheta_{k-\Delta T} - \theta^*\| - \gamma\Delta T\sum_{i=k-\Delta T}^{k-1}\frac{1-\beta}{\sqrt{\N_i}}\|(\bar\theta_i - \vartheta_i)^\top\phi_i\| - \gamma\Delta T \sum_{i=k-\Delta T}^{k-1}\!\frac{1}{\sqrt{\N_i}}\|(\vartheta_i - \theta^*)^\top\phi_i\| \\
  &\geq \Delta T \epsilon_2\|\vartheta_{k-\Delta T} - \theta^*\| - \gamma(1 - \beta)\Delta T\underbrace{\sum_{i=k-\Delta T}^{k-1}\|\bar\theta_i - \vartheta_i\|}_{W_1} - \gamma\Delta T \underbrace{\sum_{i=k-\Delta T}^{k-1}\frac{1}{\sqrt{\N_i}}\|(\vartheta_i - \theta^*)^\top\phi_i\|}_{W_2}
\end{align*}
Note that from \eqref{eq:eta2}, $\epsilon_2 > \gamma\eta(1 - \beta)$. Therefore
\begin{align}
  W_2 \geq \frac{\Delta T}{1 + \gamma\Delta T}[\epsilon_2 - \gamma\eta(1 - \beta)]\|\vartheta_{k-\Delta T} - \theta^*\| \geq 0.
  \label{eq:w2}
\end{align}
Substitute \eqref{eq:w2} into \eqref{eq:upper}, we obtain
\begin{align*}
  V_k - V_{k-\Delta T} &\leq -c_4\frac{1}{\Delta T} W_2^2\nonumber\\
                       &\leq -c_4\frac{\Delta T}{(1 + \gamma\Delta T)^2}[\epsilon_2 - \gamma\eta(1 - \beta)]^2\|\vartheta_{k-\Delta T} - \theta^*\|^2\nonumber\\
                       &\leq -c_4\frac{\Delta T}{(1 + \gamma\Delta T)^2}[\epsilon_2 - \gamma\eta(1 - \beta)]^2\lambda\gamma V_{k-\Delta T}\nonumber
\end{align*}
Thus
\begin{equation}
  \label{eq:v22}
  V_k \leq \left\{1 -c_4\frac{\Delta T}{(1 + \gamma\Delta T)^2}[\epsilon_2 - \gamma\eta(1 - \beta)]^2\lambda\gamma\right\} V_{k-\Delta T}
\end{equation}
\textbf{Case 2}: $\|\vartheta_{k-\Delta T}-\theta^*\|^2\leq \lambda\gamma V_{k-\Delta T}$

From
\begin{equation*}
  V_{k-\Delta T} = \frac{1}{\gamma}\|\theta_{k-\Delta T} - \vartheta_{k-\Delta T}\|^2 + \frac{1}{\gamma}\|\vartheta_{k-\Delta T} - \theta^*\|^2
\end{equation*}
we immediately get
\begin{equation}
  \|\theta_{k-\Delta T} - \vartheta_{k-\Delta T}\|^2 \geq (1 - \lambda)\gamma V_{k-\Delta T}
  \label{eq:31}
\end{equation}
\noindent \textbf{(a)} If $W_2 \geq \zeta\|\theta_{k-\Delta T} - \vartheta_{k-\Delta T}\|$, where $\zeta$ satisfies \eqref{eq:zeta2}, then we have
  \begin{align*}
    V_k - V_{k-\Delta T} &\leq -c_4\frac{1}{\Delta T}\zeta^2\|\theta_{k-\Delta T} - \vartheta_{k-\Delta T}\|^2\\
                         &\leq -c_4\frac{1}{\Delta T}\zeta^2(1-\lambda)\gamma V_{k-\Delta T}
  \end{align*}
  Thus
  \begin{equation}
    V_k \leq \left[1-c_4\frac{1}{\Delta T}\gamma\zeta^2(1 - \lambda)\right]V_{k-\Delta T}
    \label{eq:v23}
  \end{equation}
  
\noindent \textbf{(b)} If $W_2\leq \zeta\|\theta_{k- \Delta T} - \vartheta_{k-\Delta T}\|$, then
  \begin{align*}
    W_1 &= \sum_{i=k-\Delta T}^{k-1}\left\|\bar\theta_i - \vartheta_i\right\|\\
        &= \sum_{i=k-\Delta T}^{k-1}\Big\|\theta_{k-\Delta T} - \vartheta_{k-\Delta T} + \bar\theta_i - \theta_{k-\Delta T} - (\vartheta_i - \vartheta_{k-\Delta T})\Big\|\\
        &\geq \sum_{i=k-\Delta T}^{k-1}\|\theta_{k-\Delta T} - \vartheta_{k-\Delta T}\| - \sum_{i=k-\Delta T}^{k-1}\|\vartheta_i - \vartheta_{k-\Delta T}\| - \sum_{i=k-\Delta T}^{k-1}\|\bar\theta_i - \theta_{k-\Delta T}\| \\
        &\geq \Delta T \|\theta_{k-\Delta T} - \vartheta_{k-\Delta T}\| - \sum_{i=k-\Delta T}^{k-1}\|\vartheta_i - \vartheta_{k-\Delta T}\| - \sum_{i=k-\Delta T}^{k-1}\left\|\theta_i - \theta_{k-\Delta T} - \gamma\beta\frac{\phi_i\phi_i^\top\tilde{\theta}_i}{\N_i}\right\|\\
        &\geq \Delta T \|\theta_{k-\Delta T} - \vartheta_{k- \Delta T}\| - \Delta T \sup_{i\in[k-\Delta T, k-1]}\|\vartheta_i - \vartheta_{k-\Delta T}\| - \Delta T \sup_{i\in[k-\Delta T, k-1]}\left\|\theta_i - \theta_{k-\Delta T}\right\| \\
        &\quad - \gamma\beta\sum_{i=k-\Delta T}^{k-1}\left\|\frac{\phi_i\phi_i^\top\tilde{\theta}_i}{\N_i}\right\|\\
        &\geq \Delta T \|\theta_{k-\Delta T} - \vartheta_{k-\Delta T}\| -  \Delta T\sum_{i=k-\Delta T}^{k-2}\|\vartheta_{i+1} - \vartheta_i\| - \Delta T\sum_{i=k-\Delta T}^{k-2}\|\theta_{i+1} - \theta_i\|\\
        &\quad - \gamma\beta\sum_{i=k-\Delta T}^{k-1}\left\|\frac{\phi_i\phi_i^\top\tilde{\theta}_i}{\N_i}\right\|\\
    &\geq \Delta T \|\theta_{k-\Delta T} - \vartheta_{k-\Delta T}\| - \gamma\Delta T\sum_{i=k-\Delta T}^{k-1}\left\|\frac{\phi_i\phi_i^\top\tilde{\theta}_{i+1}}{\N_i}\right\| - \beta\Delta T\sum_{i=k-\Delta T}^{k-1}\left\|(\bar\theta_i - \vartheta_i) + \gamma\frac{\phi_i\phi_i^\top\tilde{\theta}_i}{\N_i}\right\|\\
        &\quad - \gamma\beta\sum_{i=k-\Delta T}^{k-1}\left\|\frac{\phi_i\phi_i^\top\tilde{\theta}_i}{\N_i}\right\|\\
        &\geq \Delta T \|\theta_{k-\Delta T} - \vartheta_{k-\Delta T}\| - \gamma\Delta T \sum_{i=k-\Delta T}^{k-1}\frac{1}{\sqrt{\N_i}}\|(\vartheta_i - \theta^*)^\top\phi_i\| - \gamma\Delta T\sum_{i=k-\Delta T}^{k-1}\frac{1-\beta}{\sqrt{\N_i}}\|(\bar\theta_i - \vartheta_i)^\top\phi_i\|\\
        &\quad -\beta\Delta T\sum_{i=k-\Delta T}^{k-1}\|\bar\theta_i - \vartheta_i\| - \gamma\beta\Delta T \sum_{i=k-\Delta T}^{k-1}\left\|\frac{\phi_i\phi_i^\top\tilde{\theta}_i}{\N_i}\right\| - \gamma\beta\sum_{i=k-\Delta T}^{k-1}\left\|\frac{\phi_i\phi_i^\top\tilde{\theta}_i}{\N_i}\right\|\\
        &\geq \Delta T \|\theta_{k-\Delta T} - \vartheta_{k-\Delta T}\| - \gamma\Delta T W_2 - \gamma(1 - \beta)\Delta T W_1 - \beta\Delta T W_1 - \gamma\beta(1 + \Delta T)\underbrace{\sum_{i=k-\Delta T}^{k-1}\left\| \frac{\phi_i\phi_i^\top\tilde{\theta}_i}{\N_i} \right\|}_{W_3}
  \end{align*}
  The term $W_3$ can be upper-bounded by
\begin{align*}
  W_3 &\leq \sum_{i=k-\Delta T}^{k-1}\frac{1}{\sqrt{\N_i}}\left\|(\vartheta_i - \theta^*)^\top\phi_i\right\| + \sum_{i=k-\Delta T}^{k-1}\frac{1}{\sqrt{\N_i}}\left\|(\theta_i - \vartheta_i)^\top\phi_i\right\|\\
      &\leq \sum_{i=k-\Delta T}^{k-1}\frac{1}{\sqrt{\N_i}}\left\|(\vartheta_i - \theta^*)^\top\phi_i\right\| + \sum_{i=k-\Delta T}^{k-1}\frac{1}{\sqrt{\N_i}}\left\|\phi_i^\top\left(\bar\theta_i - \vartheta_i + \gamma\beta\frac{\phi_i\phi_i^\top\tilde{\theta}_i}{\N_i}\right)\right\|\\
      &\leq W_2 + W_1 + \gamma\beta W_3
\end{align*}
Therefore, we have
\begin{align*}
  W_3 \leq \frac{1}{1 - \gamma\beta}(W_1 + W_2).
\end{align*}
Thus
\begin{align*}
  W_1 &\geq \Delta T \|\theta_{k-\Delta T} - \vartheta_{k-\Delta T}\| - \gamma\Delta T W_2 - \gamma(1 - \beta)\Delta T W_1 - \beta\Delta T W_1 - \frac{\gamma\beta(1 + \Delta T)}{1 - \gamma\beta}(W_1 + W_2)
\end{align*}
and
\begin{align*}
  W_1 &\geq \frac{\Delta T - \gamma\zeta\Delta T - \frac{\gamma\beta\zeta(1 + \Delta T)}{1 - \gamma\beta}}{1 + \gamma(1 - \beta)\Delta T + \beta\Delta T + \gamma\beta\frac{1 + \Delta T}{1 - \gamma\beta}}\|\theta_{k-\Delta T} - \vartheta_{k-\Delta T}\|\\
      &= \xi\|\theta_{k-\Delta T} - \vartheta_{k-\Delta T}\|
\end{align*}
From \eqref{eq:zeta2}, $\xi > 0$. Considering \eqref{eq:ub},
\begin{align*}
  V_k - V_{k-\Delta T} &\leq -c_3\frac{1}{\Delta T}W_1^2\\
                       &\leq -c_3\frac{1}{\Delta T}\xi^2\|\theta_{k-\Delta T} - \vartheta_{k-\Delta T}\|^2\\
                       &\leq -c_3\frac{1}{\Delta T}\xi^2(1 - \lambda)\gamma V_{k-\Delta T}
\end{align*}
which then leads to
\begin{equation}
  \label{eq:v24}
  V_k \leq \left[1 -c_3\frac{1}{\Delta T}\xi^2(1 - \lambda)\gamma\right] V_{k-\Delta T}
\end{equation}
Considering \eqref{eq:v21}, \eqref{eq:v22}, \eqref{eq:v23} and \eqref{eq:v24}, we have
\begin{equation*}
  V_k \leq (1 - \mu) V_{k-\Delta T},
\end{equation*}
where $\mu$ is defined in \eqref{eq:mu2}. Collecting terms, we obtain
\begin{equation*}
  V_k \leq \exp\left(-\mu\left\lfloor\frac{k}{\Delta T}\right\rfloor\right)V_0,
\end{equation*}
proving Theorem \ref{theo:2} (ii).
\end{proof}
\begin{remark}
  While Theorem \ref{theo:1} addressed the accelerated learning property of the HB algorithm, Theorem \ref{theo:2} addressed the same property using the NA algorithm. The difference between the two algorithms is that while HB computes gradient first and momentum next, NA does the computation in the reversed order. This reversed-order computation leads to an extra gradient term, which appears in the third equation in \eqref{eq:ht}. It is the distinction between $\bar\theta_k$ and $\theta_k$ that introduces additional challenges in proving the accelerated learning property of the NA algorithm. This required an addition of two subcases in case 2 where we have shown that in each subcase, because of the structure of the extra gradient and properties of persistent excitation, the properties of $\bar\theta_k$ still transfer to $\theta_k$, resulting in exponentially fast parameter convergence.
\end{remark}

\section{Numerical Simulations}
\label{sec:sim}

\subsection{Output Error Convergence with Constant Inputs}
We first establish the fast convergence of the output error to zero with the HB and the NA algorithms using numerical studies, which was shown analytically in \cite{Gaudio20AC}. Consider a linear regression problem with parameter $\theta^* = [20, -3, 1]^\top$. We assume that the regressor $\phi_k$ is piecewise constant, with a jump at iteration 251 from $[1, -2, 1]^\top$ to $[2, -1, -2]^\top$. We compare the three algorithms in \eqref{eq:gd}, \eqref{eq:hb} and \eqref{eq:ht}. The hyperparameters are chosen as follows:
\begin{equation}
\begin{split}
  \beta &= 0.5,\\
  \gamma &= 0.0938,\\
  \alpha &= \gamma\beta = 0.0469
\end{split}
\label{eq:hyper}
\end{equation}

\begin{figure}[!htb]
  \centering
  \includegraphics[width=0.7\linewidth]{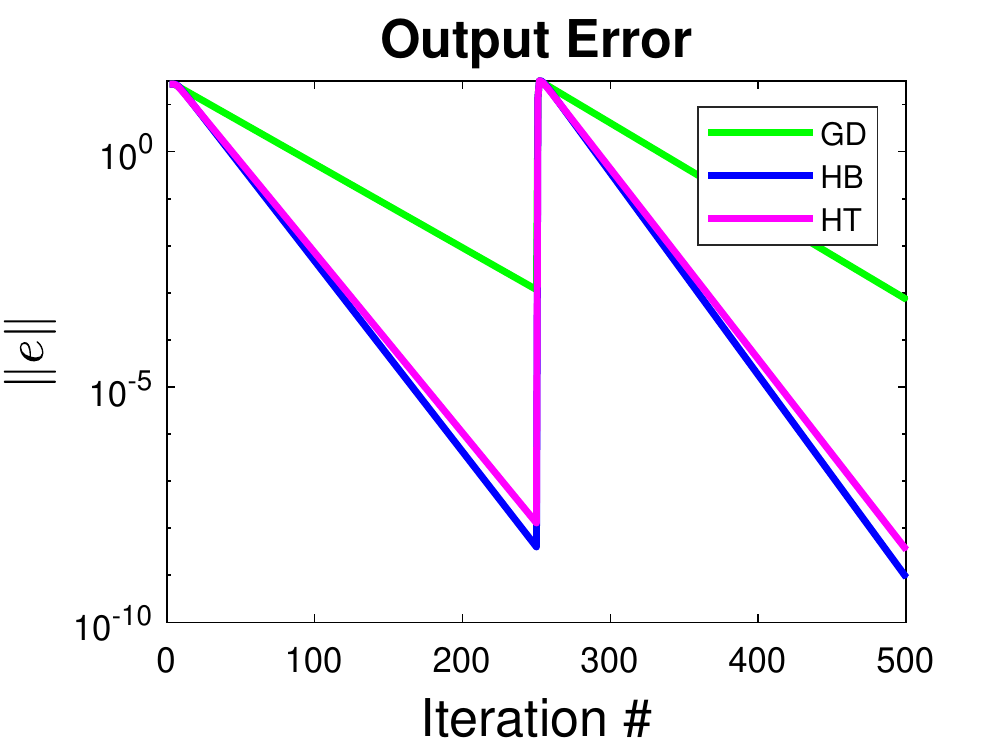}
  \caption{Output errors of the three algorithms.}
  \label{fig:2}
\end{figure}

It is clear from Fig. \ref{fig:2}, which shows the output error in log scale, that both HB and NA are significantly faster than normalized gradient descent algorithm. This corroborates the results in \cite{Polyak64}, \cite{Nesterov_1983} and \cite{Gaudio20AC}.

\subsection{Parameter Error Convergence with Persistent Excitation}
Consider the same regression problem with parameter $\theta^* = [20, -3, 1]^\top$ but now with persistently exciting regressors. The inputs for identifying parameters are defined as

$$\phi_k = [1, 2\sin(k), 2\sin(2k)]^\top.$$
We adopt the same hyperparameters as defined in \eqref{eq:hyper}. Fig. \ref{fig:1} shows parameter error convergence for the three algorithms. 

\begin{figure}[!htb]
  \centering
  \includegraphics[width=0.7\linewidth]{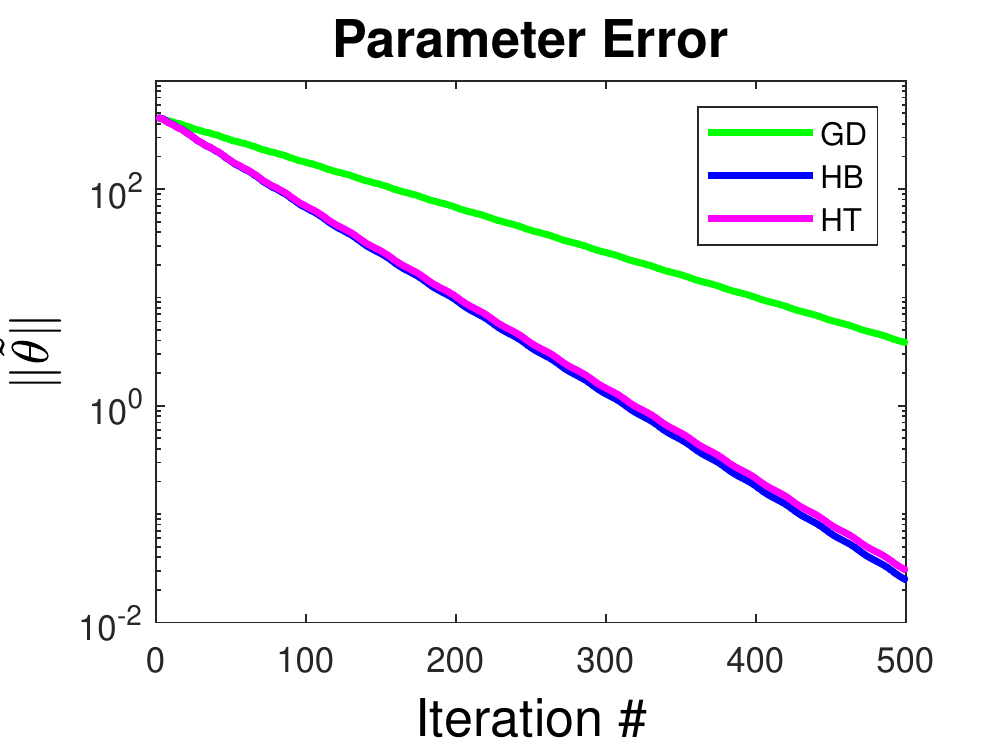}
  \caption{Parameter errors of the three algorithms.}
  \label{fig:1}
\end{figure}

From the simulations, we can see that both HB and NA have very similar convergence behavior. It is interesting to note that the two algorithms are faster than normalized GD in parameter convergence as well. This demonstrates that HT algorithms are of immense value, as they are capable of resulting in \textit{both} accelerated performance and accelerated learning.

\section{Conclusion}
\label{sec:conclusion}
Two high-order tuning algorithms, one based on Polyak's Heavy Ball method and another based on Nesterov's acceleration, have been discussed in this paper. The main results of the paper are that the parameter estimates based on the two algorithms are guaranteed to be bounded, and that if the regressors are persistently exciting, they are proved to converge to the true values. Numerical results show that these two algorithms result both in accelerated performance and accelerated learning when compared to the normalized gradient descent algorithm. These results clearly demonstrate the strong potential that these high-order tuners can have in real-time decision making. For future work, we will look into how a regularization term in the loss function can affect the overall performance of the mentioned algorithms  when disturbances and unmodeled dynamics are present.

\bibliographystyle{IEEEtran}
\bibliography{IEEEabrv,References}

\end{document}